\documentclass[12pt,a4paper,leqno]{amsart}
\usepackage{booktabs}
\usepackage{stmaryrd}
\usepackage{fourier-orns}
\usepackage{extarrows}
\usepackage{amssymb,xspace} 
\usepackage{amstext}
\usepackage{pdfsync}
\usepackage{hyperref} 
\usepackage[mathscr]{eucal} 
\theoremstyle{plain}
\usepackage{amsbsy,amssymb,amsfonts,latexsym,eucal,amscd}
\usepackage[all]{xy}

\usepackage{stackengine}
\usepackage{xcolor}

\newcommand{\bbmu}{%
  \stackengine{0pt}{\mu}{\kern0.15ex\mu}{O}{l}{F}{T}{S}%
}

\newtheorem{theorem}{Theorem}[section]

\newtheorem{lemma}[theorem]{Lemma}

\newtheorem{proposition}[theorem]{Proposition}

\newtheorem{corollary}[theorem]{Corollary}

{\theoremstyle{definition}}

{\theoremstyle{definition}}

{\theoremstyle{definition}}

{\theoremstyle{definition}\newtheorem{example}[theorem]{Example}}

{\theoremstyle{definition}\newtheorem{definition}[theorem]{Definition}}

{\theoremstyle{definition}}

{\theoremstyle{definition}\newtheorem{remark}[theorem]{Remark}}

\author{Julien Grivaux}

\address{UMR7586 -- Institut de Math\'ematiques de Jussieu-Paris Rive Gauche \\
Sorbonne Universit\'e -- 4, Place Jussieu -- case 247 -- 75252 Paris Cedex 05 \\
\url{http://www.imj-prg.fr/}}

\email{julien.grivaux@img-prg.fr}
\title{Duality, polarity and convolution in umbral calculus}

\begin{document}

\begin{abstract}
In this paper, we revisit foundations of umbral calculus using a straightforward approach based on an explicit matrix realization of binomial convolution. We construct an umbral duality of Wronskian type for rational curves in echelon form, and connect it with an umbral version of the polarity pairing. Then we extend additive convolution to the umbral setting and provide an explicit inversion formula for the corresponding deviation terms. This enables us to derive analogs of Grace and Walsh representation theorems for  finite differences.
\end{abstract}
\vspace*{1.cm}
\maketitle
\tableofcontents

\section{Introduction}

If $(a_k)_{k \geq 0}$ is a sequence of complex numbers with $a_1 \neq 0$, consider the operator
\[
\mathfrak{d} = \sum_{k=1}^{+\infty} a_k \dfrac{\partial^k}{\partial t^k}, a_1 \neq 0.
\]
Such an operator is called a delta operator, it is a nondegenerate formal differential operator that decreases the degree of polynomials by one. Assicuated with $\mathfrak{d}$ is attached a sequence of polynomials $(S_k)_{k \geq 0}$ by requiring that $S_0=1$, each $S_k$ has degree $k$, and for $k \geq 1$,
\[
\begin{cases} 
S_k(0)=1, \\
\mathfrak{d} S_k = k S_{k-1}.
\end{cases}
\]
The sequence of polynomials $(S_k)_{k \geq 0}$, first appearing in \cite{Steffensen1941}, has the remarkable property that it is a binomial sequence. This means that for any $k$, 
\[
S_k(x+y)=\sum_{j=0}^k \binom{k}{j} S_j(x)S_{k-k}(y).
\]
The correspondence from delta operators to binomial sequences is bijective, as proved in \cite{Mullin-Rota}, and serves as the main entry point to umbral calculus. We refer the reader to the seminal article \cite{Rota}, as well as \cite{Rota-Roman}. The core of umbral calculus involves the study of special sequences $(P_k)_{k \geq 0}$ of polynomials such that each $P_k$ has degree $k$ and the endomorphism $\mathfrak{d}$ of $\mathbb{C}_n[t]$ given by $\mathfrak{d}P_k=kP_{k-1}$ is a delta operator. These sequences, introduced by Sheffer in \cite{Sheffer}, bear his name. The most well-known examples of Sheffer sequences are the Appell sequences introduced in \cite{Appell}, that correspond to the case where $\mathfrak{d}$ is the usual derivation. They include many classical sequences of polynomials such that Abel, Bernoulli, Euler, Hermite and Laguerre polynomials.
\par \medskip
The algebraic structure underlying the miracles of umbral calculus is the Hopf algebra structure on $\mathbb{C}[t]$. Our first aim in this paper is to provide an alternative description of the main basic properties of umbral calculus, based on the fact that sequences of polynomials are torsors under the infinite group of lower-triangular invertible matrices. Our construction has the same key ingredient as in the article \cite{Aceto2017}, namely the use of the matrix
\[
\begin{pmatrix}
0 &  &  &  & (0) \\ 
1 & 0 &  &  &  \\ 
 & 2 & \ddots &  &  \\ 
 &  & \ddots &  &  \\ 
(0) &  &  & n & 0
\end{pmatrix} 
\]
called creation matrix in \textit{loc.\!\! cit.} However, we believe our approach offers a particularly simple and illuminating perspective to the theory of Sheffer sequences. 
\par \medskip
Our second aim is to develop a suitable duality in umbral calculus. For Appell polynomials, it was noticed (see \cite{Costabile}, \cite{Friendly}) that inverting binomial convolution could be expressed explicitly via Cramer's formula. This method gives explicit determinantal formula for inverting sequences of Appell polynomials that can be particularly useful to express complicated polynomials as determinants when the inverse sequence is more simple, like Bernoulli polynomials (see \cite{Costabile2}). Such  formulas have been extended to Sheffer sequences in various ways (see \cite{Costabile3}, \cite{Wang}). These technically demanding works, however, do not illuminate the underlying algebro-geometric structures. Our main discovery is that there exists a general duality theory, related to an umbral version of the polarity pairing. Furthermore, this duality takes an elegant form for Sheffer sequences.  Even in the classical case where the delta operator is the standard derivation, these results seem to be new. 
\par \medskip
Let us now be more explicit. Let $\theta \colon \mathbb{C} \rightarrow \mathbb{C}^{n+1}$ be a rational curve of the form 
\[
\theta(t) = \begin{pmatrix}
P_0(t) \\ 
P_1(t) \\ 
\vdots \\ 
P_n(t)
\end{pmatrix} 
\] 
where each $P_k$ has degree $k$. The dual curve is defined by 
\[
\theta^*(t) = \begin{pmatrix}
Q_0(t) \\ 
Q_1(t) \\ 
\vdots \\ 
Q_n(t)
\end{pmatrix} 
\] 
where 
\[
Q_i(-t)=(-1)^i i!(n-i)!\,\mathrm{det}(\theta(t), \mathfrak{d}\theta(t), \ldots, \mathfrak{d}^{n-1}\theta(t), e_{n-i})
\] 
up to some normalization constant. More generally, we can consider the quantity
\[
W_{\theta}^{\mathfrak{d}}(-t, v) = \mathrm{det}(\theta(t), \mathfrak{d}\theta(t), \ldots, \mathfrak{d}^{n-1}\theta(t), v)
\]
up to the same normalization constant. In the classical case, this quantity looks the Wronskian of $\theta$, except for the last column where we put an arbitrary vector, this is why we call $W_{\theta}^{\mathfrak{d}}$ the umbral Wronskian.
\par \medskip 
The main ingredient to understand the duality in the classical case is the polarity paring on $\mathbb{C}_n[t]$, defined by
\[
\langle A, B \rangle_n = \sum_{k=0}^n (-1)^{n-k} A^{(k)} B^{(n-k)}.
\]
This pairing is of particular importance because it is invariant under the action of $\mathrm{SL}_2(\mathbb{C})$ on $\mathbb{C}_n[t]$ given by
\[
\left[\begin{pmatrix}
a & b \\ 
c & d
\end{pmatrix}.P\right](z) =(cz+d)^n P \left(\dfrac{az+b}{cz+d}\right).
\] 
It appeared in the literature as the only joint bilinear invariant for pairs of binary forms, see \cite[Chap. XI]{Grace_Young} and \cite[\S 5]{Kung_Rota}. It can be generalized to the umbral setting by putting
\[
\langle A, B \rangle_n^{\mathfrak{d}} = \sum_{k=0}^n (-1)^{n-k} \mathfrak{d}^{k}A . (\mathfrak{d}^*)^{n-k} B
\]
where $\mathfrak{d}^*$ is defined by $\mathfrak{d}^*P=-(\mathfrak{d}P^{\dagger})^{\dagger}$ and $Q^{\dagger}(t)=Q(-t)$, a crucial point being that this quantity remains constant. A first summary of our main results is as follows:
\begin{theorem}
Let $\mathfrak{d}$ be a delta operator on $\mathbb{C}_n[t]$, and $\eta=(S_k)_{0 \leq k \leq n}$ be the associated binomial curve. Then : 
\begin{enumerate}
\item[(i)] $\langle S_a, S_b^{\dagger}\rangle_n^{\mathfrak{d}} = a!b! \delta_{a+b,n}$.
\item[(ii)] If $\theta=(P_k)_{0 \leq k \leq n}$ and $\theta^*=(Q_k)_{0 \leq k \leq n}$ is the dual curve, then $\langle P_a, Q_b^{\dagger} \rangle_n^{\mathfrak{d}}=a!b!\delta_{a+b, n}$.
\item[(iii)] The binomial curve $\eta$ is self dual.
\item[(iv)] If $\theta$ is a Sheffer curve attached to $\eta$, $\theta^*$ is its inverse (for binomial convolution) and $n! W_{\theta}^{\mathfrak{d}}(t, \eta(s))=\theta_n^*(t+s)$.
\end{enumerate}
\end{theorem}
In the classical case, \[
\eta(s)=
\begin{pmatrix}
1 \\ 
s \\ 
\vdots \\ 
s^n
\end{pmatrix}.
\]
and for $t=0$ we get the determinantal representations of Appell polynomials given in \cite{Friendly} and \cite{Costabile2}. We provide another example of a determinantal formula, which is a far reaching generalization of the Vandermonde determinant : 
\begin{theorem}
Let $\mathfrak{d}$ be a delta operator on $\mathbb{C}_n[t]$, and $\eta=(S_k)_{0 \leq k \leq n}$ be the associated binomial curve. Denote by $a_1, \ldots, a_n$ the roots of $S_n$. 
\par \medskip
Let $\theta(t)=(P_k(t))_{0 \leq t \leq n}$ where each $P_k$ has degree $k$. If $C$ is the product of the leading coefficients of the $P_k$'s, then : 
\par \medskip
\begin{enumerate}
\item[(i)] For any $x_0, \ldots, x_n$ in $\mathbb{C}$, 
\[
\mathrm{det}\,(\theta(x_0), \ldots, \theta(x_n))=  C \prod_{i < j} (x_j-x_i).
\]
\item[(ii)] For any $t$ in $\mathbb{C}$,
\[
\mathrm{det}\,(\theta(t+a_1), \ldots, \theta(t+a_n), v)=  C \prod_{i < j} (a_j-a_i) \times n! W_{\theta}^{\mathfrak{d}}(-t, v).
\]
\end{enumerate}
\end{theorem} 
Recall that the polarity pairing allows to define another interesting structure on polynomials, called additive convolution or Walsh convolution, introduced implicitly in the article \cite{Walsh}. It is defined as
\[
A \boxplus_n B(z)=\frac{1}{n!}\langle A(t), B(z-t)\rangle_n = \frac{1}{n!} \sum_{k=0}^n \partial^k_t A(0) \partial^{n-k}_t B(z).
\]
Extending additive convolution to the umbral setting is straightforward, replacing the usual polarity pairing by the umbral polarity pairing. It is defined by
\[
A \boxplus_n^{\mathfrak{d}} B(z)=\frac{1}{n!}\langle A(t), B(z-t)\rangle_n^{\mathfrak{d}} = \frac{1}{n!} \sum_{k=0}^n {\mathfrak{d}}^k_t A(0) {\mathfrak{d}}^{n-k}_t B(z).
\]
For finite differences, such convolution has already been introduced in \cite{PS}. Using umbral additive convolution, we obtain an  elegant reformulation of duality for Sheffer sequences : 
\begin{theorem}
Let $\mathfrak{d}$ be a delta operator on $\mathbb{C}_n[t]$, and $\eta=(S_k)_{0 \leq k \leq n}$ be the associated binomial curve. Then the map $v \mapsto v \star \eta(t)$ is an algebra morphism between the group of Sheffer curves attached to $\eta$ endowed with binomial convolution and $(\mathbb{C}_n[t], \boxplus_n^{\mathfrak{d}})^{\times}$.
\end{theorem}
In the classical case, this means that inverting Appell sequences corresponds to inverting them with respect to additive convolution.
\par \medskip
At this point, some explanation is necessary. Classical umbral calculus deals with infinite sequence of polynomials, while we stayed most of the time in the realm of the finite dimensional setting. The reason is simple: the duality pairing and all its variants (Wronskians, additive convolution) depend on some level $n$ that cannot be avoided. We have fully taken this point in consideration and went even further: we prove that a delta operator on $\mathbb{C}_n[t]$ (which we call an umbral structure of level $n)$ is entirely determined up to an $n^{\mathrm{th}}$ root of unity by the top degree term $S_n$ of its binomial sequence, and vice versa. This leads to the notion of an umbral structure of level $n$, which is a delta operator on $\mathbb{C}_n[t]$ modulo an $n^{\mathrm{th}}$ root of unity. It turns out that many umbral notions we considered for delta operators on $\mathbb{C}_n[t]$ like umbral wronskians, umbral pairings or umbral additive convolution only depend on the underlying umbral structure. In this way, we see that a single degree $n$ polynomial is enough to encode the whole construction. 
\par \medskip
To further investigate the interplay between binomial sequences and additive convolution, we introducte the notion of deviation polynomial. If $\mathfrak{d}$ is an umbral structure of level $n$, the deviation polynomial $R_n$ is the inverse of the binomial polynomial $S_n$ with respect to the additive convolution $\boxplus_n$. This polynomial allows to express the umbral pairing as well as the umbral additive convolution in terms of the corresponding classical operations, making its understanding essential. To this end, we prove an inversion formula, this time in the framework of infinite sequences of polynomials, which is more natural.
\begin{theorem}
Let $(S_n)_{n \geq 0}$ be a binomial sequence attached to a delta operator $\mathfrak{d}=\varphi(\partial / \partial t)$ where $\varphi \in \mathbb{C}[[z]]$, $\varphi(0)=0$, $\varphi'(0)\neq 0$.
If $(A_n)_{n \geq 0}$ is the Appell sequence whose generating series is
\[
\sum_{n=0}^{+\infty} \frac{A_n(t)}{n!} z^n=\frac{\exp(tz)}{z\mathrm{D}_{\log}(\varphi(z))}
\] then the deviation polynomial $R_n$ is given by
\[
R_n(t) =\left(\mathfrak{d} (\partial / \partial t)^{-1} \right)^{n} A_n(t) = \underbrace{\frac{\mathfrak{d} t^{n+1}}{n+1} \boxplus_n \ldots \boxplus_n \frac{\mathfrak{d} t^{n+1}}{n+1}}_{n \, \textrm{times}} \boxplus_n A_n(t).
\]
\end{theorem}
It would be very interesting to know if this identity has a probabilistic interpretation (see \cite{Markus}).
\par \medskip
For the last part of our study, let us highlight that until now, the base field could have been any characteristic zero field. However, the story takes on a flavor of complex analysis. Indeed, in the classical case, the polarity pairing is related to a celebrated theorem of Grace. To state it, recall that a circular region is an open or closed connected subset of the complex plane whose boundary is a circle or a line.
\begin{theorem}[{Grace theorem \cite{Grace1902}, see \cite[Thm. 3.4.1]{Rahman}}]
Let $A$ and $B$ be two $n$-apolar polynomials of degree $n$ Then every circular domain containing all roots of $A$ contains at least one root of $B$. 
\end{theorem}
It is a primary tool to study root localization of complex polynomials. An equivalent version, due to Walsh, runs as follows:
\begin{theorem}[Walsh representation theorem {\cite[Thm. VI]{Walsh}, \cite[Thm. 3.4.1c and Thm. 5.3.1]{Rahman}}] \label{Walsh}
If $P$ and $Q$ are polynomials of degree $n$ and if $\Omega$ is a  circular domain of $\mathbb{C}$ containing all roots of $Q$, then each root of $P \boxplus_n Q$ is of the form $z+\omega$ where $z$ is a root of $P$ and $\omega$ lies in $\Omega$. 
\end{theorem}
It is natural to adapt these theorems to the umbral setting. This is achievable provided the roots of the deviation polynomials are controlled. We explain how to do that in the case of finite differences. 
\begin{theorem}
Let us consider the delta operator $\Delta$, which is attached to the binomial sequence $((t)_n)_{n \geq 0}$. Then 
all roots of the deviation polynomial $R_n(t)$ lie in the vertical strip $\mathrm{Re}(z)=-\dfrac{n-1}{2}\cdot$
\end{theorem}
As a corollary, we get:
\begin{theorem}[Umbral Grace theorem for $\Delta$]
Let $P$ and $Q$ be two polynomials of degree $n$ such that  $\langle P, Q \rangle_n^{\Delta} =0$. Then 
\[
V(P) \cap \left(V(Q)+\frac{n-1}{2}\right)  \neq \emptyset.
\]
\end{theorem}

\begin{theorem}[Umbral Walsh theorem for $\Delta$]
Assume that $P$ and $Q$ are two polynomials of degree $n$. Then 
\[
V(P \boxplus_n^{\mathfrak{d}} Q) \subseteq V(P)+V(Q)-\left(\dfrac{n-1}{2}\right)\cdot
\]
\end{theorem}
This last theorem is somewhat complementary to the results of \cite{PS} and \cite{Leake}, that focus on real-rooted polynomials. 
\par \medskip
Finally a few philosophical remarks. Writing an article on the foundations of umbral calculus without being a combinatorialist in not easy. Furthermore, we have used neither exponentials nor generating series except in \S \ref{inversion}, which is certainly sacrilegeous. We hope for the indulgence of the experts on this matter. Our primary goal was to understand the	 algebraic structures behind the localization of the roots of some polynomials and their relation to stability theory, in the spirit of \cite{Sokal} and \cite{BB}. We hope that the results presented in this article can convince interested readers that umbral calculus and the algebraic structures underlying it are an efficient tool to tackle some problems related to complex polynomials.

\section{Umbral calculus in finite height}
All matrices will be of size $(n+1)\times (n+1)$, with row and column indices in $\llbracket 0, n \rrbracket$. We denote by $\mathrm{T}_{n+1}(\mathbb{C})$  (resp. $\mathrm{U}_{n+1}(\mathbb{C})$ the group of inversible (resp. unipotent) lower-triangular matrices.

\subsection{Binomial convolution}
Let $R$ be any commutative $\mathbb{Q}$-algebra. To any finite sequence $(a_k)_{0 \leq k \leq n}$ of elements of $R$, we associate the matrix $M(a)$ in $\mathrm{M}_{n+1}(R)$ given by
\[
(M(a))_{ij}=\binom{i}{j}a_{i-j},
\]
Let $N$ be the lower-triangular $(n+1) \times (n+1)$ nilpotent matrix given by $N_{ij}=i \delta_{i-j-1}$. 
\begin{lemma}
The commutant of $N$ is a commutative subalgebra of $\mathrm{M}_{n+1}(R)$ consisting of matrices of the type $M(a)$.
\end{lemma}
\begin{proof}
The matrix $N$ is similar to a principal Jordan block so its commutant is $\mathbb{C}[N]$. Then we remark that for $0 \leq p \leq n+1$, $N^p$ is proportional to $M(e_p)$.
\end{proof}

\begin{definition} The binomial convolution of to sequences finite sequences $a=(a_k)_{0 \leq k \leq n}$ and $b=(b_k)_{0 \leq k \leq n}$ is the sequence $c=(c_k)_{0 \leq k \leq n}$ defined by
\[
c_m=\sum_{k=0}^m \binom{m}{k} a_k b_{n-k}.
\]
We denote this by $c=a \star b$. 
\end{definition}

\begin{lemma}
The map $a \mapsto M(a)$ gives an isomorphism of commutative algebras between finite sequences endowed with binomial convolution and the commutant of $N$. Besides, for any sequences $a$ and $b$, $M(a).b=a\star b$.
\end{lemma}

\begin{proof}
$
(M(a)M(b))_{m0}=\sum_{k=0}^m \binom{m}{k}a_{m-k}\binom{k}{0}b_k=(a \star b)_m$.
\end{proof}
By a slight abuse of language, we will say that the commutant of $N$ is the convolution algebra of $R$. 
\begin{definition} The functor attaching to any $\mathbb{Q}$-algebra $R$ the group of invertible matrices in the convolution algebra of $R$ (that is matrices of the form $M(a)$ for $(a_i)_{0 \leq i \leq n+1}$ in $R^{n+1}$ such that $a_0 \in R^{\times}$) defines an affine algebraic group $\mathscr{C}onv$ over $\mathbb{Q}$, which is a subgroup of $\mathrm{T}_{n+1}$.
\end{definition}

\begin{remark}
We have $e(t+a)=e(a) \star e(t)$. Hence $e(a)^{-1}=e(-a)$, which gives Pascal inversion formula for $a=1$.
\end{remark}

\subsection{Vandermonde curves}
In combinatorics, a sequence of polynomials is a sequence $(P_k)_{k \geq 0}$ such that each $P_k$ has degree $k$. We will consider a finite version of this concept, which we call a Vandermonde curve. 
\par \medskip
\begin{definition} Let $n$ be a nonnegative integer. A Vandermonde curve with values in $\mathbb{C}^{n+1}$ is a map $\theta \colon \mathbb{C}  \rightarrow \mathbb{C}^{n+1}$ of the form
\[
\theta(t) = \begin{pmatrix}
P_0(t) \\ 
P_1(t) \\ 
\vdots \\ 
P_n(t)
\end{pmatrix} 
\]
where each $P_k$ is a polynomial of degree $k$.
\end{definition}

\begin{remark}
We could replace in this definition the complexe line $\mathbb{C}$ by the affine line $\mathbb{A}^1$, that is not specifying any origin in the source of the curve. Conceptually, this viewpoint is perfectly legitimate. However, the use of an affine coordinate $t$ (hence the choice of a basepoint) is essential to perform calculations and lies at the heart of umbral calculus, because binomial sequences need an origin. We will therefore work on $\mathbb{C}$, but we will emphasize when notions involving a Vandermonde curve $\theta$ only depend only on the underlying affine curve $\theta^{\mathrm{aff}} \colon \mathbb{A}^1 \rightarrow \mathbb{C}$.
\end{remark}

The set of Vandermonde curves is a torsor under the action of the group $\mathrm{T}_{n+1}(\mathbb{C})$. Associated with the choice of origin on $\mathbb{A}^1$ is a distinguished Vandermonde curve, namely
\[
e(t)=\begin{pmatrix}
1 \\ 
t \\ 
\vdots \\ 
t^n
\end{pmatrix} 
\]
that allows to identify the set of Vandermonde curves with $\mathrm{T}_{n+1}(\mathbb{C})$.
\begin{definition}
If $\theta$ is a Vandermonde curve, the one-parameter subgroup $U_{\theta} \colon \mathbb{C} \rightarrow \mathrm{U}_{n+1}(\mathbb{C})$ is defined by 
$
\theta(t+a)=U_{\theta}(a). \theta(t).
$
We consider $U_{\theta}$ as an element of $U_{n+1}(\mathbb{C}[a])$.
\end{definition}
\begin{remark}
We have $e(t+a)=e(a) \star e(t) = M(e(a)).e(t)$. Hence, if $\theta(t)=T.e(t)$, $\theta(t+a)=TM(e(a)).e(t)$ so $U_{\theta}(a)=TM(e(a))T^{-1}$.
\end{remark}

\begin{definition}
If $\theta=(P_k(t))_{0 \leq k \leq n}$ is a Vandermonde curve, the operator $\Delta_{\theta}$ is the nilpotent endomorphism of $\mathbb{C}_n[t]$ of degree $-1$ having $N$ as matrix in the basis $(P_0(t), \ldots, P_n(t))$. 
\end{definition}
This means that 
\[
\begin{cases}
\textrm{For}\,\,1 \leq k \leq n, 
\Delta_{\theta}(P_k(t))=kP_{k-1}(t),\\ 
\Delta(P_0(t))=0.
\end{cases}
\]
\begin{proposition} \label{devvv}
If $\theta$ is a Vandermonde curve and $v$ is a vector in $\mathbb{C}^{n+1}$ such that $v_0 \neq 0$, then $v \star \theta$ is a Vandermonde curve and $\Delta_{v \star \theta}=\Delta_{\theta}$.
\end{proposition}

\begin{proof}
We have $(v \star theta(t))_m=\sum_{k=0}^m \binom{m}{k} v_{m-k} P_{k}(t)$ and
\begin{align*}
\Delta_{\theta}\left(\sum_{k=0}^m \binom{m}{k} v_{m-k} P_{k}(t) \right)
&= \sum_{k=1}^m k\binom{m}{k} v_{m-k} P_{k-1}(t) \\
&= \sum_{k=1}^m m\binom{m-1}{k-1} v_{m-k} P_{k-1}(t) \\
&= m \sum_{k=0}^{m-1} \binom{m-1}{k} v_{m-1-k} P_{k}(t)
\end{align*}
\end{proof}
whence the result.
\subsection{Sheffer and binomial curves}

\begin{definition}
A delta operator is an endomorphism $\mathfrak{d}$ of $\mathbb{C}_n[t]$ which is translation invariant, and such $\mathfrak{d}(1)=0$ and $\mathfrak{d}(t) \neq 0$.
\end{definition}
It is easy to see that delta operators are exactly endomorphisms of $\mathbb{C}_n[t]$ of the form
\[
a_1 \frac{\partial}{\partial t} + a_2 \frac{\partial^2}{\partial t^2} + \ldots + a_n \frac{\partial^n}{\partial t^n}
\]
where $a_1 \neq 0$. This is because translation-invariant endomorphisms of $\mathbb{C}_n[t]$ are entirely determined by their image on $t^n$. We now define Sheffer and binomial curves.
\begin{definition} $ $ \par
\begin{enumerate}
\item[--] A Vandermonde curve $\theta$ is a Sheffer curve if the endomorphism $\Delta_{\theta}$ is a delta operator.
\item[--] A Sheffer curve $\eta$ is a binomial curve if $\eta(0)=e_0$.
\end{enumerate}
\end{definition}

\begin{proposition} \label{building}
The following properties hold:
\begin{enumerate}
\item[(i)] A Vandermonde curve $\theta$ is a Sheffer curve if $U_{\theta}$ belongs to $\mathscr{C}onv(\mathbb{C}[a])$.
\item[(ii)] A binomial curve is a Vandrmonde curve $\eta$ satisfying the equation ${\eta}(a+b)={\eta}(a) \star {\eta}(b)$.
\item[(iii)] The map ${\eta} \rightarrow \Delta_{\eta}$ gives a one to one correspondance between binomial curves and delta operators.
\item[(iv)] If $\theta$ is a Sheffer curve and $\eta$ is the associated binomial curve, then
\[
\begin{cases}
\theta(t+a)=\theta(t) \star {\eta}(a) \\
U_{\theta}(a)=M(\eta(a)).
\end{cases}
\]
\item[(v)] Sheffer curves attached to a binomial curve $\eta$ are curves of the form $v \star \eta(t)$
 where $v$ is an invertible sequence in the convolution algebra of $\mathbb{C}$.
\end{enumerate}
\end{proposition}

\begin{proof} $ $ 
\par
\begin{enumerate}
\item[(i)] The matrix of $\Delta_{\theta}$ in the basis $\theta(t+a)$ is $U_{\theta}(a) N U_{\theta}(a)^{-1}$. Therfore, $\Delta_{\theta}$ commutes with translation if and only if $U_{\theta}(a)$ commutes with $N$ for all $a$, which proves (i).
\item[(ii)] If $\eta$ satisfies the convolution equation, $\eta(a)=\eta(a) \star \eta(0)$. Since $\eta(a)$ is invertible in the convolution algebra of $\mathbb{C}$ (because its first component is a nonzero constant), it follows that $\eta(0)=e_0$. Besides, $\eta(a+b)=M(\eta(a)).\eta(b)$ so $U_{\theta}(a)=M(\eta(a))$ belongs to the convolution algebra of $\mathbb{C}[a]$. Conversely, if $\eta$ is a binomial curve, $U_{\eta}(a)=M(\vartheta(a))$ for some sequence of polynomials $\vartheta(a)$. This means that 
$
\eta(t+a)=\vartheta(a)\star \eta(t)
$
Taking $t=0$ gives $\eta=\vartheta$, so $\eta(t+a) = \eta(a)\star \eta(t)$.
\item[(iii)]
If $\mathfrak{d}$ is a delta operator, for any $k$ in $\llbracket 1, n \rrbracket$, $\mathfrak{d} \colon \mathbb{C}_{k}[t] \rightarrow \mathbb{C}_{k-1}[t]$ is surjective with one-dimensional kernel, since it is the composition of a unipotent endomorphism with the derivation $\partial / \partial t$. Thus, if $\mathfrak{d}=\Delta_{\theta}$, then $\theta=(S_k)_{0 \leq k \leq n}$ is uniquely determined by the system of equations
\[
\begin{cases}
S_0=1 \\
\mathfrak{d}S_k=kS_{k-1} \,\, \textrm{and}\,\, S_k(0)=0
\end{cases}
\]
\item[(iv)] Since $\theta$ is a Sheffer curve, we can write $U_{\theta}(a) = M(\vartheta(a))$ for some sequence $\vartheta(a)$, yielding $\theta(t+a) = \vartheta(a) \star \theta(t)$. Hence $\vartheta(a) = \theta(t+a) \star \theta(t)^{-1}$. By proposition \ref{devvv}, $\vartheta$ is a Vandermonde curve and taking $t=0$, $\Delta_{\vartheta}=\Delta_{\theta}$. Since $\vartheta(0) = e_0$, $\vartheta$ is a binomial curve. By the uniqueness in (iii), $\vartheta = \eta$. Consequently, $U_{\theta}(a)=M(\eta(a))$ and $\theta(t+a)=\theta(t) \star \eta(a)$.
\item[(v)] This follows directly from (iv) by taking $t=0$.
\end{enumerate}
\end{proof}

\begin{proposition}[Umbral Taylor formula] \label{taylor}
If $P$ is a polynomial in $\mathbb{C}_n[t]$ and $\eta=(Q_k)_{0 \leq k \leq n}$ is the binomial curve attached to $\mathfrak{d}$, then 
\[
P(t+a)=\sum_{k=0}^n  \frac{\mathfrak{d}^k P(t)}{\partial t^k} \times \frac{Q_k(a)}{k!} \cdot
\]
\end{proposition}
\begin{proof}
Without loss of generality, we can assume that $\deg P=n$. If $(P_k)_{0 \leq k \leq n}$ is the Sheffer sequence defined by $(n)_{k} P_{n-k}=\mathfrak{d}^{k}P$, then Sheffer identity (proposition ref{building}(iv)) gives
\[
P(t+a)=P_n(t+a)=\sum_{k=0}^n \binom{n}{k} P_{n-k}(t)Q_k(a)
\]
whence the result.
\end{proof}

\begin{remark}
The element $U_{\theta}$, the notion of Sheffer identity (proposition \ref{building} (iv) as well as Taylor's formula (proposition \ref{taylor}) depend only on the underlying affine Vandermonde curve $\theta^{\mathrm{aff}}$. The notion of bininomial curve is vectorial because it depends on a base point. 
\end{remark}

\begin{example} $ $ \par
\begin{enumerate}
\item[(i)] 
If $\mathfrak{d}$ is the standard derivation, the corresponding binomial curve is $(t^k)_{0 \leq k \leq n}$. Sheffer curves correspond to the classical Appell polynomials (\textit{see} \cite{Appell}, \cite{Friendly}). 
\item[(ii)]
If $\mathfrak{d}(P)=P'(t+a)$, the corresponding binomial curve is given by Abel polynomials $t(t-na)^{n-1}$.
\item[(iii)]
If $\mathfrak{d}$ is the forward difference $\Delta$ given by 
\[
\Delta P(t)=P(t+1)-P(t),
\] then the binomial curve is given by the descending factorials $((t)_k)_{0 \leq k \leq n}$. 
\item[(iv)]
If $\mathfrak{d}$ is the backward difference $\Lambda$ given by
\[
\Lambda P(t)=P(t)-P(t-1),
\]
the binomial curve is given by ascending factorials $(t^{(k)})_{0 \leq k \leq n}$. 
\end{enumerate}
\end{example}

\subsection{Umbral structures}

\begin{proposition}
For any polynomial $S$ of degree $n$, there exists a delta operator, unique up to an $n$-th root of unity, whose binomial sequence has $S$ as highest-degree term.
\end{proposition}

\begin{proof}
We can assume $S$ monic and prove existence and uniqueness of $\mathfrak{d}$ of the form $\partial / \partial t + \textrm{lower order terms}$.  Let
\[
\mathfrak{d}=\partial / \partial t + \lambda_2 \partial^2 / \partial t^2 + \ldots + \lambda_{n} \partial^{n} / \partial t^{n}
\]
We claim that $S$ is given by a formula of the form
\[
S(t) = t^n + a_1 \lambda_2 t^{n-1}+ (a_2\lambda_3 + P_1(\lambda_2)) t^{n-2}+ \ldots + (a_{n-1} \lambda_n + P_{n-2}(\lambda_2, \ldots, \lambda_{n-1})) 
\]
where the $a_i$ are nonzero coefficients and the $P_i$ are polynomials without constant coefficients. Indeed, by induction, if 
\[
S(t)=t^n + c_1 t^{n-1}+ \ldots + c_n
\]
we can apply the induction to $\mathfrak{d}S/n$ for the differential operator
\[
\partial / \partial t + \lambda_2 \partial^2 / \partial t^2 + \ldots + \lambda_{n-1} \partial^{n-1} / \partial t^{n-1}
\]
so
\[
\mathfrak{d}S =  nt^{n-1}+ b_1 \lambda_2 t^{n-2}+ \ldots + b_{n-2} (\lambda_{n-1} + P_{n-3}(\lambda_2, \ldots, \lambda_{n-2})).
\]
Solving the system
\begin{align*}
&(n-1)_1 c_1 + (n)_2 \lambda_2 = b_1 \lambda_2\\
&(n-2)_1 c_2 + (n-1)_2 \lambda_2 c_1 +  (n)_3 \lambda_3 = b_2 \lambda_3 +P_1(b_2) \\
&(n-3)_1 c_3 + (n-2)_2 \lambda_2 c_2 + (n-1)_3 \lambda_3 c_4 + (n)_4 \lambda_4 = b_3 \lambda_4 +P_2(b_2,b_3) \\
\vdots
\end{align*}
gives the required expression for $S$, except for the nonvanishing of the $a_i$'s. To do this, we consider the particular case
\[
\mathfrak{d}=\partial / \partial t + \partial^p / \partial t^p \cdot
\]
Then $S(t)=t^n + a_{p-1}t^{n-p+1} + \textrm{lower order terms}$. If $a_{p-1}=0$, then $S(t)=t^n + R(t)$ where $\deg \, R \leq n-p$. In this case, $R^{(n-p+1)}=0$ so
\begin{align*}
\mathfrak{d}^{n-p+1}S(t)&= \left(\mathrm{id}+ \partial^{p-1} / \partial t^{p-1}\right)^{n-p+1} \partial^{n-p+1} / \partial t^{n-p+1} (t^n + R(t)) \\
&=\left(\mathrm{id}+ \partial^{p-1}/ \partial t^{p-1}\right)^{n-p+1} (n)_{n-p+1} t^{p-1}
\end{align*}
so $\mathfrak{d}^{n-p+1} S(0) \neq 0$ which is a contradiction.
\end{proof}

\begin{definition}
An umbral structure of level $n$ is the datum of a delta operator modulo $n$-th roots of unity on $\mathbb{C}_n[t]$.
\end{definition}

\begin{example}
Let us describe the umbral correspondence $\mathfrak{d} / \mu_n \leftrightarrow S_n$ for $n=3$. It is given by
\[
\begin{cases}
\zeta^{-1} \left(\dfrac{\partial}{\partial t} + \alpha \dfrac{\partial^2}{\partial t^2} + \beta \dfrac{\partial^3}{\partial t^3} \right) \rightarrow \zeta^3 (t^3+6\alpha t^2+6(2\alpha^2-\beta)t) \\
\lambda(t^3 + at^2 + bt) \rightarrow \lambda^{-1/3} \left(\dfrac{\partial}{\partial t} + \dfrac{a}{6}  \dfrac{\partial^2}{\partial t^2} + \dfrac{a^2-3b}{18} \dfrac{\partial^3}{\partial t^3} \right) 
\end{cases}
\]
The binomial sequence is given by
\[
\begin{cases}
S_0(t)=1 \\
S_1(t)=\zeta t \\
S_2(t) = \zeta^2 (t^2-2\alpha t) \\
S_3(t) = \zeta^3 (t^3+6\alpha t^2+6(2\alpha^2-\beta)t)
\end{cases}
\]
\end{example}
%
%
%

\subsection{Group-theoretic aspects}
Let $\mathfrak{d}$ be a delta operator and $\eta$ be the associated binomial curve. Then the set of Vandermonde curves can be identified with $\mathrm{T}_{n+1}(\mathbb{C})$.

\begin{proposition} Let $\eta$ be a binomial curve. After trivializing the $\mathrm{T}_{n+1}(\mathbb{C})$-torsor of Vandermonde curves at $\eta$, the following properties hold:
\begin{enumerate}
\item[--]
The set of binomial curves is a subgroup $N$ of $\mathrm{T}_{n+1}(\mathbb{C})$; it consists of matrices that normalize $\mathscr{C}onv(\mathbb{C})$ in $U_{n+1}(\mathbb{C})$ and fix $e_0$.
\item[--]
The set of Sheffer curves whose delta operator is $\mathfrak{d}$ is $\mathscr{C}onv(\mathbb{C})$, it is a subgroup of $\mathrm{T}_{n+1}(\mathbb{C})$.
\item[--] Sheffer curves form a subgroup of $\mathrm{T}_{n+1}(\mathbb{C})$, which is the subgroup $\mathscr{C}onv(\mathbb{C}) \rtimes N$.
\end{enumerate}
\end{proposition}

\begin{proof}
Let $\vartheta(t)=T.\eta(t)$. Since $\eta(t+a)=\eta(a) \star \eta(t) = M(\eta(a)). \eta(t)$, we have $\vartheta(t+a)=TM(\eta(a)).\eta(t)=TM(\eta(a))T^{-1}.\vartheta(t)$ so $U_{\vartheta}(a)=TM(\eta(a))T^{-1}$. If $\vartheta$ normalizes $\mathscr{C}onv(\mathbb{C})$, then for all $a$, $U_{\vartheta}(a)$ belongs to $\mathscr{C}onv(\mathbb{C})$, which implies that $U_{\vartheta}$ belongs to $\mathscr{C}onv(\mathbb{C}[a])$. Conversely, if $\vartheta$ is a binomial curve, $T M(\eta(a))T^{-1}$ belongs to $\mathscr{C}onv(\mathbb{C})$ for all $a$. This proves that $\mathrm{Ad}(T)$ maps the convolution algebra of $\mathbb{C}$ to itself. This mapping is injective: if  $T M(a) T^{-1}=I_n$, $T$ itself belongs to the convolution algebra but since $T$ fixes $e_0$, $T=I_n$. Hence $\mathrm{Ad}(T)$ is an isomorphism of the convolution algebra of $\mathbb{C}$.
\end{proof}

\subsection{Generating series}
For some computations it is necessary to work with infinite sequences of polynomials and to use generating series. We recall the following results:

\begin{proposition}
A sequence $(S_n)_{n \geq 0}$ of polynomials is a binomial sequence if and only if
\[
\sum_{n=0}^{+\infty} \dfrac{S_n(t)}{n!}z^n = \exp(t \phi(z))
\]
for some formal power series $\phi$ with $\phi(0)=0$ and $\phi'(0) \neq 0$. Furthermore, the associated delta operator is given by
\[
\mathfrak{d}=\phi^{-1}(\partial / \partial t).
\]
\end{proposition}

\begin{proposition}
A sequence $(A_n)_{n \geq 0}$ of polynomials is an Appell sequence if and only if 
\[
\sum_{n=0}^{+\infty} \dfrac{A_n(t)}{n!}z^n = a(z)\exp(t z)
\]
for some formal power series $a$ such that $a(0)=0$. Moreover, the bijection between Appell sequences and invertible power series is a group isomorphism.
\end{proposition}

\section{Additive convolution}

Throughout this section, we fix a delta operator $\mathfrak{d}$, and denote by $(S_k)_{0 \leq k \leq n}$ the corresponding binomial sequence.

\subsection{Polarity pairing}
For any polynomial $P$, define $P^{\dagger}(t)=P(-t)$. Furthermore, we define an involution on differential operators with constant coefficients (in particular, on umbral structures) by
\[
D^{*}(P)=-(D(P^{\dagger}))^{\dagger}.
\]
The binomial sequence associated with $\mathfrak{d}^{*}$ is $((-1)^k S_k^{\dagger})_{0 \leq k \leq n}$. 
\begin{remark}
The involution $D \mapsto D^{*}$ is independent of the basepoint $0$ and can be defined on differential operators with constant coefficients on the affine line. Indeed, these differential operators admit $\partial^p /\partial t^p$ as a basis, where $t$ is any affine coordinate with unit derivative. In this basis, $D \mapsto D^{*}$ is the multiplication by $(-1)^p$ on each vector $\partial^p /\partial t^p$.
\end{remark}

\begin{definition}
For $A$, $B$ in $\mathbb{C}_n[t]$, we define their umbral polarity pairing as
\[
\langle A, B \rangle^{\mathfrak{d}}_n = \sum_{p=0}^n (-1)^{n-p}  \,\mathfrak{d}^p A \times (\mathfrak{d}^{*})^{n-p} B
\]
If $\mathfrak{d}$ is the derivation, we write $\langle \star, \star \rangle^{\mathfrak{d}}_n=\langle \star, \star \rangle_n$.
\end{definition}

This quantity is, a priori, an element of $\mathbb{C}_n[t]$, but it turns out to be constant. The main properties of this pairing are summarized in the following result : 

\begin{theorem} \label{masse}
The pairing $\langle \star, \star \rangle^{\mathfrak{d}}_n$ takes constant values, is translation invariant, and depends only on the umbral structure attached to $\mathfrak{d}$. Besides, 
\begin{enumerate}
\item[(i)] $\langle B, A \rangle^{\mathfrak{d}}_n = (-1)^n \langle A, B \rangle_n^{\mathfrak{d}^*}=\langle A^{\dagger}, B^{\dagger} \rangle^{\mathfrak{d}}_n$.
\item[(ii)] If $\deg B=n-k$, $\langle A, B \rangle^{\mathfrak{d}}_n=\langle \mathfrak{d}^k A, B \rangle^{\mathfrak{d}}_{n-k}$.
\item[(iii)] If $\deg A=n-k$, $\langle A, B \rangle^{\mathfrak{d}}_n=(-1)^k \langle A, (\mathfrak{d}^{*})^k B \rangle^{\mathfrak{d}}_{n-k}$.
\item[(iv)] $\langle S_{n}(t-\alpha), Q \rangle^{\mathfrak{d}}_n = n! Q(\alpha)$.
\item[(v)] $\langle S_p, S_q^{\dagger} \rangle^{\mathfrak{d}}_n= p! q! \delta_{p+q, n}.$
\item[(vi)] $\langle v \star \eta(t), w \star \eta(-t) \rangle^{\mathfrak{d}}_n= n! \,v \star w.$
\item[(vii)] For any differential operator with constant coefficients $D$, we have $\langle DA, B \rangle_n^{\mathfrak{d}}=-\langle A,  D^* B\rangle_n^{\mathfrak{d}}$.
\end{enumerate}
\end{theorem}

\begin{proof}
Properties (i), (ii), and (iii) are straightforward. We now  prove (iv). For this, let $(Q_k)_{0 \leq k \leq n}$ the Sheffer sequence associated with $\mathfrak{d}^{*}$ such that $Q_n=Q$. Then
\begin{align*}
\langle S_n(t-\alpha), Q(t) \rangle^{\mathfrak{d}}_n &= \sum_{p=0}^n (n)_{n-p} (-1)^{n-p}  S_{n-p}(t-\alpha) (n)_p Q_p(t) \\
&= n!\sum_{p=0}^n \binom{n}{p} (-1)^{n-p} S_{n-p}^{\dagger}(\alpha-t)  Q_p(t) \\
&=n! Q_n(\alpha) \\
&=n! Q(\alpha).
\end{align*}
Applying (ii), we find that 
\[
\langle S_{n-p}(t-\alpha), Q(t) \rangle^{\mathfrak{d}}_n=(n)_{p} (\mathfrak{d}^*)^p Q(\alpha).
\]
This proves, in particular, that the pairing takes constant values and is invariant under translation. Taking $\alpha=0$ and $Q=S_q^{\dagger}$, we get 
\begin{align*}
\langle S_{n-p}, S_q^{\dagger} \rangle^{\mathfrak{d}}_n &= (n)_p (\mathfrak{d}^*)^p S_q^{\dagger} (0) \\
&=(n)_p (q)_p (-1)^p S_{q-p}^{\dagger}(0) \\
&=  (n)_p (q)_p (-1)^p \delta_{q-p, n} \\
&= n! \delta_{q-p, n}.
\end{align*}
To obtain (vi), it suffices to take $v=e_{n-p}$ and $w=e_{n-q}$; the statement then reduces to (v).
\end{proof}

\begin{remark}
Umbral pairing is a purely affine notion and is independant of the base point $0$.
\end{remark}

\subsection{Differential operators}

Let $\mathcal{D}_n$ denote the algebra of differential operators on $\mathbb{C}$ with constant coefficients of order at most $n$; it is the same as the algebra of translation-invariant endomorphisms of $\mathbb{C}_n[t]$. 

\begin{enumerate}
\item[--] Via the umbral structure $\mathfrak{d}$, there is a natural algebra isomorphism
\[
\mathcal{D}_n \simeq \mathbb{C}_n[t] /(\mathfrak{d}^{n+1})\cdot
\]
\item[--] There is also a linear isomorphism between $\mathcal{D}_n$ and $\mathbb{C}_n[t]$ given by $D \mapsto D(S_n)$, where $S_n$ is the highest-degree component of the binomial sequence attached to $\mathfrak{d}$.
\end{enumerate}

\begin{definition}
Umbral additive convolution on $\mathbb{C}_n[t]$ is the algebra structure $\boxplus_n^{\mathfrak{d}}$ obtained by transporting the commutative algebra structure of $\mathcal{D}_n$ via the map $D \mapsto D(S_n)$. If $\mathfrak{d}$ is the derivation, we write $\boxplus_n^{\mathfrak{d}}=\boxplus_n$. 
\end{definition}

Concretely, this means that $D_1(S_n)  D_2(S_n)=(D_1D_2)(S_n)$. Any translation-invariant operator $T$ acting on $\mathbb{C}_n[t]$ is of the form
\[
Q \mapsto P \boxplus_n^{\mathfrak{d}} Q,
\]
where $P=T(S_n)$.
\begin{proposition} \label{misc} Umbral additive convolution $\boxplus_n^{\mathfrak{d}}$ satisfies the following properties : 

\begin{enumerate}
\item[(i)] The operation $\boxplus_n^{\mathfrak{d}}$ is bilinear, symmetric and associative, with neutral element $S_n$.
\item[(ii)]
$
(n)_k \,S_{n-k} \boxplus_n^{\mathfrak{d}} Q = \mathfrak{d}^k Q.
$
\item[(iii)] For any differential operator with constant coefficients $D$, we have $D(P \boxplus_n^{\mathfrak{d}} Q)=D(P) \boxplus_n^{\mathfrak{d}} Q=  P \boxplus_n^{\mathfrak{d}} D(Q)$.
\item[(iv)] If $\mathrm{deg} P = n-k$, $(n)_k P \boxplus_n^{\mathfrak{d}} Q = P \boxplus_n^{\mathfrak{d}} \mathfrak{d}^k Q$.
\item[(v)] Invertible elements for the additive convolution are polynomials of degree $n$. 
\item[(vi)] $ n! P \boxplus_n^{\mathfrak{d}} Q(z)=\langle P(t), Q(z-t)\rangle_n^{\mathfrak{d}}$.
\item[(vii)] $n! P \boxplus_n^{\mathfrak{d}} Q (z)=\sum_{k=0}^n \mathfrak{d}^{k}P(0) \times \mathfrak{d}^{n-k}Q(z)$.
\item[(viii)] $(P \boxplus_n^{\mathfrak{d}} Q)^{\dagger} 
= (-1)^n P^{\dagger} \boxplus_n^{\mathfrak{d}^*} Q^{\dagger}$.
\end{enumerate}
\end{proposition}

\begin{proof}
Properties (i) to (iv) are straightforward. Property (vi) follows from the fact that the polynomials $S_n(t-\alpha)$ generate $\mathbb{C}_n[t]$.
If $P=S_n(t-\alpha)$, both sides are equal to $n! Q(z)$. To prove (vii), we write
\begin{align*}
(\mathfrak{d}^*)^{n-k} (P(z-t))_{|t=0}&=(\mathfrak{d}^*)^{n-k} (P(z+t)^{\dagger})_{|t=0}\\
&=(-1)^{n-k} (\mathfrak{d}^{n-k} P(z+t))^{\dagger}_{|t=0} \\
&=(-1)^{n-k} \mathfrak{d}^{n-k} P(z+t)_{|t=0} \\
&=(-1)^{n-k} \mathfrak{d}^{n-k} P(z)
\end{align*}
For property (viii), we have
\begin{align*}
(P \boxplus_n^{\mathfrak{d}} Q)^{\dagger}(z) &= \langle P(t), Q(-z-t) \rangle_n^{\mathfrak{d}} \\
&= (-1)^n \langle P(-t), Q(-z+t) \rangle_n^{\mathfrak{d}^*} \\
&= (-1)^n \langle P^{\dagger}(t), Q^{\dagger}(z-t) \rangle_n^{\mathfrak{d}^*} \\
&= (-1)^nP^{\dagger} \boxplus_n^{\mathfrak{d}^*} Q^{\dagger}(z).
\end{align*}
\end{proof} 

\begin{remark}
Additive convolution is a vectorial notion and depends on the basepoint $0$.
\end{remark}

\subsection{Deviation polynomials}
\begin{definition}
The deviation polynomial $R_n$ is the inverse of $S_n$ with respect to the standard additive convolution $\boxplus_n$; that is, $R_n \boxplus_n S_n = t^n$.
\end{definition}

\begin{proposition} \label{devv}
The classical and umbral polarity pairings and convolutions are related by the following formulas : 
\[
\begin{cases}
\langle \langle P, Q \rangle \rangle_n^{\mathfrak{d}}=\langle R_n \boxplus_n P, Q \rangle_n = (-1)^n \langle P, R_n^{\dagger} \boxplus_n Q \rangle_n \\
P \boxplus_n^{\mathfrak{d}} Q = R_n \boxplus_n P \boxplus_n Q
\end{cases}
\]
\end{proposition}

\begin{proof}
We begin with the second identity. Considered as endomorphism of the variable $Q$, boths sides are translation-invariant endomorphisms. Thus it suffices to prove that they agree on $S_n$, which is straightforward. For the first identities we replace $Q$ by $Q^{\dagger}$ and evaluate $t=0$, which yields the first equality. Writing
\[
R_n \boxplus_n P \boxplus_n Q^{\dagger} = P \boxplus_n ((-1)^n R_n^{\dagger} \boxplus_n Q)^{\dagger}
\]
gives the second equality.
\end{proof}
This allows to give an explicit formula for $R_n$ in terms of the umbral pairing : 

\begin{proposition} \label{toboggan}
$n! R_n(\alpha)=\langle t^n, (\alpha-t)^n \rangle^{\mathfrak{d}}_n$.
\end{proposition}

\begin{proof}
\[
\langle t^n, (t-\alpha)^n \rangle^{\mathfrak{d}}_n =(t^n \boxplus_n^{\mathfrak{d}} t^n) (\alpha) = R_n(\alpha).
\]
\end{proof}

Just as a binomial sequence can be reconstructed from its highest-degree term, the sequence of deviation polynomials of lower-degree can be reconstructed from $R_n$ : 

\begin{theorem} \label{apdrift}
The deviation polynomials satisfy the following properties : 
\begin{enumerate}
\item[--] Recurrence relation : 
\[
R'_k=R_{k-1} \boxplus_{k-1} \mathfrak{d} t^k.
\]
\item[--] If we extend the umbral structure $\mathfrak{d}$ at level $n+1$, there exists an Appell sequence $(A_k)_{0 \leq k \leq n}$ such that
\[
R_k(t) = 
\underbrace{\frac{\mathfrak{d} t^{k+1}}{k+1} \boxplus_k \ldots \boxplus_k \frac{\mathfrak{d} t^{k+1}}{k+1}}_{k \, \textrm{times}} \boxplus_k A_k(t) = \left(\mathfrak{d}(\partial / \partial t)^{-1} \right)^k A_k(t).
\]
\end{enumerate}
\end{theorem}

\begin{remark}
Different choices of extension of $\mathfrak{d}$ differ by a multiple of $\partial^{n+1} / \partial t^{n+1}$, which give different possible values for $A_n$. 
\end{remark}

\begin{remark}
Since $\mathfrak{d}(1)=0$, $\mathfrak{d}(\partial / \partial t)^{-1}$ is well defined on $\mathbb{C}_n[t]$.
\end{remark}

\begin{proof}

\par \medskip
We applying $\mathfrak{d}$ to the equality $R_k(t) \boxplus_k S_k(t)=t^k$. This gives
\[
k R_k(t) \boxplus_{k} S_{k-1}(t)=\mathfrak{d} t^{k}
\]
that is
\[
R'_k(t) \boxplus_{k-1} S_{k-1}(t)= \mathfrak{d} t^{k}.
\]
Taking the convolution with $R_{k-1}(t)$ gives the fist relation.
Since $\mathfrak{d} t^{k+1}$ is invertible, $A_k$ is uniquely defined by the second property, and we get 
\begin{align*}
R'_k(t) &= \underbrace{\frac{\mathfrak{d} t^{k+1}}{k+1} \boxplus_k \ldots \boxplus_k \frac{\mathfrak{d} t^{k+1}}{k+1}}_{k \, \textrm{times}} \boxplus_k A'_k(t) \\
&= \frac{1}{k} \left( \underbrace{\frac{\mathfrak{d} t^{k+1}}{k+1} \boxplus_k \ldots \boxplus_k \frac{\mathfrak{d} t^{k+1}}{k+1}}_{k \, \textrm{times}}\right)' \boxplus_{k-1} A'_k(t) \\
&=\left(\frac{\mathfrak{d}t^k}{k} \boxplus_k   \underbrace{\frac{\mathfrak{d} t^{k+1}}{k+1} \boxplus_k \ldots \boxplus_k \frac{\mathfrak{d} t^{k+1}}{k+1}}_{k-1 \, \textrm{times}}\right) \boxplus_{k-1} A'_k(t) \\
&= \frac{\mathfrak{d}t^k}{k} \boxplus_{k-1} \frac{1}{k} \left(   \underbrace{\frac{\mathfrak{d} t^{k+1}}{k+1} \boxplus_k \ldots \boxplus_k \frac{\mathfrak{d} t^{k+1}}{k+1}}_{k-1 \, \textrm{times}}\right)' \boxplus_{k-1} A'_k(t) \\
&= \ldots \\
&=  \underbrace{\frac{\mathfrak{d}t^k}{k} \boxplus_{k-1} \ldots \boxplus_{k-1}  \frac{\mathfrak{d}t^k}{k}}_{k \, \textrm{times}} \boxplus_{k} A'_k(t)
\end{align*}
so $A'_k(t)=k A_{k-1}(t)$. It remains to prove that $\dfrac{\mathfrak{d}t^{k+1}}{k+1} \boxplus_k (\star) = \mathfrak{d}(\partial / \partial t)^{-1}$. To prove it, let $P$ be in $\mathbb{C}_k[t]$ and $\widehat{P}$ be a primitive of $P$. Then
\begin{align*}
\dfrac{\mathfrak{d}t^{k+1}}{k+1} \boxplus_k P &= \mathfrak{d}t^{k+1} \boxplus_k \dfrac{\partial\widehat{P}/\partial t}{k+1} \\
&=\mathfrak{d}t^{k+1} \boxplus_{k+1} \widehat{P} \\
&=t^{k+1} \boxplus_{k+1} \mathfrak{d} \widehat{P} \\
&=\mathfrak{d} \widehat{P} \\
&=\mathfrak{d}(\partial / \partial t)^{-1} (P).
\end{align*}

\end{proof}

\begin{example} \label{abel}
Let us consider the case of Abel polynomials, that is $S_n(t)=t(t-an)^{n-1}$ and $\mathfrak{d}(P)=P'(t+a)$. In this case, we claim that
\[
A_n(t)=\sum_{k=0}^{n-1} (-1)^{k} (n)_k  a^{k} t^{n-k} \\
\]
so that
\[
R_n(t)=\sum_{k=0}^{n-1} (-1)^{k} (n)_k  a^{k} (t-an)^{n-k}.
\]
This means that 
\[
t(t-an)^{n-1} \boxplus_n (t+an)^{n} \boxplus_n \left( \sum_{k=0}^n 	(-1)^{k} (n)_k a^{k} t^{n-k} \right) =  t^n.
\]
We do it by direct calculation. Using the formulas
\[
\begin{cases} 
(t+\alpha)^n \boxplus_n P(t)=P(t-\alpha) \\
n (t+\alpha)^{n-1} \boxplus_n P(t) = P'(t-\alpha)
\end{cases}
\]
we get
\begin{align*}
&t(t-an)^{n-1} \boxplus_n (t+an)^{n} \boxplus_n \left( \sum_{k=0}^n 	(-1)^{k} (n)_k a^{k} t^{n-k} \right) \\
&= \left((t-an)^{n}+an(t+an)^{n-1}\right) \boxplus_n \left( \sum_{k=0}^n (-1)^{k} (n)_k a^{k} (t-an)^{n-k} \right) \\
&=\sum_{k=0}^n 	(-1)^{k} (n)_k a^{k} t^{n-k}- \sum_{k=0}^n (-1)^{k+1} (n)_{k+1} a^{k+1} t^{n-(k+1)}=t^n.
\end{align*}
\end{example}

\subsection{An inversion theorem} \label{inversion}

\begin{theorem} \label{lagrange}
Let $(S_n)_{n \geq 0}$ be a binomial sequence associated with a delta operator $\mathfrak{d}=\varphi(\partial / \partial t)$ where $\varphi \in \mathbb{C}[[z]]$ satisfies $\varphi(0)=0$, $\varphi'(0)\neq 0$.
\par \medskip
If $(A_n)_{n \geq 0}$ is the Appell sequence defined by  $A_n(t)=\left(\mathfrak{d} (\partial / \partial t)^{-1} \right)^{-n} R_n(t)$, then 
\[
\sum_{n=0}^{+\infty} \frac{A_n(t)}{n!} z^n=\frac{\exp(tz)}{z\mathrm{D}_{\log}(\varphi(z))}\cdot
\]
\end{theorem}

\begin{proof}
We have 
\[
\sum_{n=0}^{+\infty} \frac{S_n(t)}{n!} z^n=\exp(t \phi(z)),
\]
where $\phi=\varphi^{-1}$.
Let $B_n(t)=\left(\mathfrak{d} (\partial / \partial t)^{-1} \right)^{n} S_n(t)$. Then $(B_n)_{n \geq 0}$ is an Appell sequence, since
\begin{align*}
B'_n(t)&=\partial / \partial t \left(\mathfrak{d} (\partial / \partial t)^{-1} \right)^{n} S_n(t) \\
& = \left(\mathfrak{d} (\partial / \partial t)^{-1} \right)^{n-1} \mathfrak{d}S_n(t) \\
& = n \left(\mathfrak{d} (\partial / \partial t)^{-1} \right)^{n-1} S_{n-1}(t) \\
& = n B_{n-1}(t).
\end{align*}
The sequences $(A_n)_{n \geq 0}$ and $(B_n)_{n \geq 0}$ are inverse to each other. Indeed, since $(B_n)_{n \geq 0}$ is an Appell sequence, inverting it within the group of Appell sequences is the same as inverting it with respect to additive convolution. We now compute: 
\begin{align*}
\left(\mathfrak{d}(\partial / \partial t)^{-1} \right)\left(\sum_{n=0}^{+\infty} \frac{B_n(t)}{n!} z^n\right) &= \sum_{n=0}^{+\infty} \frac{\left(\mathfrak{d}(\partial / \partial t)^{-1} \right)^{n+1} S_n(t)}{n!} z^n \\
&=\sum_{n=0}^{+\infty} \left(\frac{\phi^{-1}(\partial / \partial t)}{\partial / \partial t} \right)^{n+1} [w^n](\exp(t \phi(w)) z^n \\
&= \sum_{n=0}^{+\infty} [w^n] \left(\left( \frac{w}{\phi(w)}\right)^{n+1} \exp(t \phi(w))\right) z^n.
\end{align*}
If $\phi^{-1}(z)=\sum_{n=1}^{+\infty} c_n z^n$, then by Lagrange inversion formula,
\[
[w^n] \left( \frac{w}{\phi(w)}\right)^{n+1} = (n+1)c_{n+1}. 
\]
Hence, setting $t=0$, we obtain
$
\sum_{n=0}^{+\infty} (n+1) c_{n+1} z^n = (\phi^{-1})'(z).
$
On the other hand, 
\begin{align*}
\left(\mathfrak{d}(\partial / \partial t)^{-1} \right)\left(\sum_{n=0}^{+\infty} \frac{B_n(t)}{n!} z^n\right) &= \left(\mathfrak{d}(\partial / \partial t)^{-1} \right) (b(z)\exp(tz)) \\
&=b(z) \frac{\phi^{-1}(z)}{z} \exp(tz).
\end{align*}
Hence $b(z)=z \dfrac{ (\phi^{-1}(z))'}{\phi^{-1}(z)}$, which yields the result.
\end{proof}

\begin{example} \label{yoo} $ $ \par
\begin{enumerate}
\item[(i)] $\mathfrak{d}=\Delta$. Then $S_n(t)=(t)_n$ and $\varphi(z)=e^{z}-1$ so
\[
\frac{1}{z \mathrm{D}_{\mathrm{log}}(\varphi(z))}
=\frac{e^z-1}{ze^{z}}
=\frac{1-e^{-z}}{z} \cdot
\]
Hence 
\begin{align*}
\frac{1-e^{-z}}{z} \exp(tz) &= \frac{\exp(z(t-1))-\exp(tz)}{z} \\
&= \sum_{n=1}^{+\infty} \frac{t^n-(t-1)^n}{n!} z^{n-1} \\
&=\sum_{n=0}^{+\infty} \frac{t^{n+1}-(t-1)^{n+1}}{(n+1)} \times \frac{z^{n}}{n!}
\end{align*}
so $A_n(t)=\dfrac{\Lambda t^{n+1}}{n+1} \cdot$ where $\Lambda$ is the backward difference operator.
\item[(ii)] $\mathfrak{d}(P)=P'(t+a)$. Then $S_n(t)=t(t-an)^{n-1}$ are the Abel polynomials and $\varphi(z)=z\exp(az)$ so 
\[
z\mathrm{D}_{\mathrm{log}}(\varphi(z))=1+az.
\]
Hence $A_n(0)=(-1)^n n! a^n$ so
\begin{align*}
A_n(t)&=\sum_{k=0}^n \binom{n}{k}A_k(0) t^{n-k} \\
&=\sum_{k=0}^n \binom{n}{k} (-1)^k k ! a^k t^{n-k} \\
&=\sum_{k=0}^n (-1)^k (n)_k a^k t^{n-k}
\end{align*}
which gives the formula obtained in example \ref{abel}.
\item[(iii)] $ \mathfrak{d}=\log \left(1+\partial / \partial t \right)$. The corresponding binomial sequence is the Touchard (or exponential) polynomials $T_n(t)$ given by 
\[
T_n(t)=\sum_{k=0}^n S(n, k) t^k
\] 
where $S(n, k)$ are the Stirling numbers of the second kind. 
Then 
\[
z \mathrm{D}_{\mathrm{log}}(\varphi(z)) = \dfrac{z}{(1+z) \log(1+z)}
\]
and since
\begin{align*}
\dfrac{(1+z) \log(1+z)}{z} &= (1+z) \times \sum_{n=0}^{+\infty} (-1)^{n} \frac{z^n}{n+1} \\
&= 1 + \sum_{n=1}^{+\infty} (-1)^{n+1} \frac{z^n}{n(n+1)}
\end{align*}
so $A_0=1$ and $A_n(0)=(-1)^{n+1} (n-1)!/(n+1)$.
\end{enumerate}
\end{example}

\section{Duality for Vandermonde curves}
We take the same setting as in the preceding section; namely an umbral structure $\mathfrak{d}$ of level $n$. We denote by $\eta=(S_k)_{0 \leq k \leq n}$ the binomial sequence attached to $\mathfrak{d}$. 

\subsection{Wronskian and duality}

\begin{definition}
If $\theta=(P_k)_{0 \leq k \leq n}$ is a Vandermonde curve, we define its umbral Wronskian $W^{\mathfrak{d}}_{\theta}$ in $\mathbb{C}_n(t) \otimes_{\mathbb{C}}(\mathbb{C}^{n+1})^*$ by
\[
W^{\mathfrak{d}}(-t, v)= \frac{\theta(t) \wedge \mathfrak{d} \theta(t) \wedge \ldots \wedge \mathfrak{d}^{n-1} \theta (t) \wedge v}{\theta(t) \wedge \mathfrak{d} \theta(t) \wedge \ldots \wedge \mathfrak{d}^{n} \theta((t)}\cdot
\]
\end{definition}

\begin{proposition} \label{key}
For any $T$ in $\mathrm{T}_{n+1}(\mathbb{C})$, if $\theta = T. \eta$, 
\[
W^{\mathfrak{d}}_{\theta}(t, v)=\frac{1}{n!} \sum_{k=0}^n \binom{n}{k} (T^{-1}v)_{n-k} S_k(t).
\]
\end{proposition}

\begin{proof}
Let $A(\theta,t)$ be the matrix whose columns are $\theta(t), \mathfrak{d} \theta(t), \ldots, \mathfrak{d}^n \theta(t)$, and write $\theta=T.\eta$ for some $T$ in $\mathrm{T}_{n+1}(\mathbb{C})$. Since the delta operator $\mathfrak{d}$ vanishes on constants, we have 
$
A(\theta, t)=T. A(\eta, t).
$
As
\[
A(\eta, t)_{ij}=\mathfrak{d}^j S_i(t)=(i)_j S_{i-j}(t) = j! \binom{i}{j}  S_{i-j}(t)
\]
we see that $A(\eta, t)=M(\eta(t)) D$, where $D$ is the diagonal matrix given by $D_{ii}=i!$. By Cramer's rule, 
\begin{align*}
W^{\mathfrak{d}}_{\theta}(t, v)&= (A(\theta, -t)^{-1}v)_n \\
&=(D^{-1}M(\eta(-t))^{-1} T^{-1}v)_n \\
&=\frac{1}{n!} (M(\eta(t)) T^{-1}v)_n \\
&= \frac{1}{n!} (T^{-1}v  \star \eta(t))_n.
\end{align*}
\end{proof}

\begin{corollary} \label{uw} The umbral Wronskian satisfies the following properties:
\begin{enumerate}
\item[(i)] $n! W^{\mathfrak{d}}_{\theta}(t, \theta(s))=S_n(t+s).$
\item[(ii)] Foy any $\ell$ in $\llbracket 0,n \rrbracket$, $\langle P_{\ell}, W^{\mathfrak{d}}_{\theta}(-t, v)\rangle^{\mathfrak{d}}_n=v_{\ell}.$
\end{enumerate}
\end{corollary}

\begin{proof}
The first identity follows directly from proposition \ref{taylor} since 
\[
\theta(s)=\sum_{k=0}^n \frac{\mathfrak{d}^k \theta(-t)}{k!} S_k(s+t).
\]
For the second identity, we compute : 
\begin{align*}
\langle P_{\ell}(t), W^{\mathfrak{d}}_{\theta}(-t, v)\rangle^{\mathfrak{d}}_n &=\langle (T.\eta(t))_{\ell},  W^{\mathfrak{d}}_{\theta}(-t, v)\rangle^{\mathfrak{d}}_n \\
&=\frac{1}{n!} \sum_{j=0}^n  T_{\ell j} \sum_{k=0}^n \binom{n}{k} \sum_{q=0}^{n}   T^{-1}_{(n-k) q} \langle S_j(t), S_{k}(-t)\rangle^{\mathfrak{\delta}}_n v_q  \\
&= \frac{1}{n!}\sum_{j=0}^n  T_{\ell j} \sum_{k=0}^n \binom{n}{k} \sum_{q=0}^{n}   T^{-1}_{(n-k) q} \delta_{j+k, n} j! k!  v_q  \\
&=\sum_{j=0}^n  T_{\ell j} \sum_{q=0}^{n}   T^{-1}_{j q} v_q \\
&=\delta_{\ell q} v_q \\
&=v_{\ell}.
\end{align*}
\end{proof}

\begin{definition}
If $\theta=(P_k)_{0 \leq q \leq n}$ is a Vandermonde curve, we define its $\mathfrak{d}$-dual $\theta^*:=(Q_k)_{0 \leq k \leq n}$ as follows:
\[
Q_k(t)=k! (n-k)! \times W_{\theta}^{\mathfrak{d}}(t, e_{n-k}).
\]
\end{definition}

\begin{theorem}
If $\theta=(P_k)_{0 \leq k \leq n}$ is a Vandermonde curve,then : 
\begin{enumerate}
\item[(i)] The dual curve $\theta^*=(Q_k)_{0 \leq k \leq n}$ is also a Vandermonde curve.
\item[(ii)] 
$
\langle P_a, Q_{b}^{\dagger}\rangle^{\mathfrak{d}}_n= a! b! \delta_{a+b, n}.
$
\item[(iii)] Duality is an involution given by $T.\eta \mapsto (D^{-1}T^{-1}D). \eta$ where $D$ is the diagonal matrix with binomial coefficients $\binom{n}{k}$. 
\end{enumerate}
\end{theorem}

\begin{proof}
We compute explicitly $Q_{\ell}$ using proposition \ref{key}:
\[
Q_{\ell}(t)=\frac{\ell! (n-\ell)!}{n!} \sum_{k=0}^n \binom{n}{k} (T^{-1}e_{n-\ell})_{n-k} S_k(t).
\]
This establishes (iii) and (i); and (ii) follows from corollary \ref{uw}.
\end{proof}

\subsection{Sheffer and binomial curves}

We identify the set of Sheffer curves attached to $\mathfrak{d}$ with the invertible elements in the binomial convolution algebra $\mathbb{C}^{n+1}$ via the map $v \mapsto v \star \eta(t)$.

\begin{theorem} The following properties hold:
\begin{enumerate}
\item[(i)] If $\theta$ is a Sheffer curve attached to $\mathfrak{d}$, $\theta^*$ is its inverse for binomial convolution.
\item[(ii)] $n! W^{\mathfrak{d}}_{\theta}(t, \eta(s))=\theta_n^*(t+s)$.
\item[(iii)] The map $v \mapsto v \star \eta(t)$ is an algebra morphism between $(\mathbb{C}^{n+1}, \star)^{\times}$ and $(\mathbb{C}_n[t], \boxplus_n^{\mathfrak{d}})^{\times}$.
\end{enumerate}
\end{theorem}

\begin{proof}
If $\theta(t)=u \star \eta(t)=M(u).\eta(t)$, we have 
\begin{align*}
n!W^{\mathfrak{d}}_{\theta}(t, v) &=((M(u))^{-1}v)\star \eta(t))_n \\
&=((M(u^{-1}))v)\star \eta(t))_n \\
&= (u^{-1}\star v\star \eta(t))_n.
\end{align*}
Hence 
\begin{align*}
k! (n-k)! W^{\mathfrak{d}}(t, e_{n-k})&=\dfrac{1}{\binom{n}{k}} (u^{-1}\star e_{n-k}\star \eta(t))_n \\
&=\dfrac{1}{\binom{n}{k}}\left((u^{-1}\star \eta(t))\star e_{n-k}\right)_n \\
&=(u^{-1} \star \eta(t))_k.
\end{align*}
Besides, 
\begin{align*}
n!W^{\mathfrak{d}}_{\theta}(t, \eta(s)) &= (u^{-1}\star \eta(s) \star \eta(t))_n \\
&=  (u^{-1}\star \eta(t+s))_n \\
&=\theta^*_n(t+s).
\end{align*}
Finally,
\begin{align*}
\left((u \star \eta(t)) \boxplus_n^{\mathfrak{d}} (v \star \eta(t))\right) (z)&= \frac{1}{n!}\langle u \star \eta(t), v \star \eta(z-t) \rangle_n^{\mathfrak{d}} \\
&= \frac{1}{n!}\langle u \star \eta(t), (v \star \eta(z))\star \eta(-t)) \rangle_n^{\mathfrak{d}} \\
&=u \star v \star \eta(z)
\end{align*}
thanks to theorem \ref{masse} (v).
\end{proof}

We now consider binomial curves associated with an umbral structure that is not necessarily the one used to compute the Wronskian.

\begin{theorem}
Let $(\mathfrak{d}, \eta, (S_k)_{0 \leq k \leq n})$ and $(\mathfrak{s}, \vartheta, (T_k)_{0 \leq k \leq n})$ be two umbral structures. The dual of the binomial curve $\vartheta$ with respect to the umbral structure $\mathfrak{d}$ is $\left(\dfrac{\mathfrak{s}^{n-k} S_n(t)}{(n)_k}\right)_{0 \leq k \leq n}$. The corresponding Wronskian is given by
\[
W_{\vartheta}^{\mathfrak{d}}(t, v)=\frac{1}{n!} \sum_{p=0}^n \dfrac{\mathfrak{s}^p S_n(t)}{p!} v_p.
\]
\end{theorem}

\begin{remark}
If $\mathfrak{s}=\mathfrak{d}$ and $\vartheta = \eta$, the curve $\eta$ is self-dual and 
\[
n!W_{\eta}^{\mathfrak{d}}(t, v)=\sum_{p=0}^n \binom{n}{p} S_{n-p}(t)v_p.
\]
\end{remark}


\begin{proof}
$W_{\eta}(t, v)$ is entirely determined on the vectors $v=\eta(s)$ for all complex numbers $s$, where it equals $S_n(t+s)/n!$. On the other hand,
\[
\frac{1}{n!} \sum_{p=0}^n \dfrac{\mathfrak{s}^p S_n(t)}{p!} S_p(t)= S_{n}(s+t)/n!
\]
thanks to Taylor umbral formula.
\end{proof}

\subsection{Appell polynomials}

Appell polynomials correspond to the case where $\mathfrak{d}$ is the standard derivation. In this case, the associated binomial curve is $e(t)=(t^k)_{0 \leq k \leq n}$.

\begin{theorem}
The map $v \mapsto v \star e(t)$ gives an isomorphism between $(\mathbb{C}^{n+1}, \star)^{\times}$ and $(\mathbb{C}_n[t], \boxplus)^{\times}$.
\end{theorem}

\begin{theorem}
Let $(A_n)_{n \geq 0}$ an Appell sequence and let $(B_n)_{n \geq 0}$ be the inverse Appell sequence. Then, 
\[
B_n(t+x)= n! \, W(t, e_{x}). 
\]
\end{theorem}

Setting $t=0$, we get:
\begin{corollary}\emph{(cf. \cite[\S 4]{Friendly}, \cite{Costabile})} \label{cocorico}  If $a_k=A_k(0)$,
\[
B_n(t)=\frac{1}{a_0^{n+1}}\begin{vmatrix}
a_0 &  &  &  &  &1\\ 
a_1 & a_0 &  &  &  &t\\ 
a_2 & \binom{2}{1} a_1 & a_0 &  &  & t^2 \\[3pt] 
a_3 & \binom{3}{1} a_2 & \binom{3}{2}a_1 & a_0 & & \vdots\\ 
\vdots & \vdots &  &  & a_0 & t^{n-1} \\ 
a_n & \binom{n}{1} a_{n-1} & \binom{n}{2}a_{n-2} & \ldots & \binom{n}{n-1} a_1 & t^n
\end{vmatrix} 
\]
\end{corollary}

If we apply this to the Appell sequence $\left( \dfrac{\Delta t^{n+1}}{n+1}\right)_{n \geq 0}$ whose inverse is the sequence $(B_n(t))_{n \geq 0}$ of Bernoulli polynomials, then $a_k=\dfrac{1}{k+1}$ and corollary \ref{cocorico} gives the approach of Bernoulli polynomials given in \cite{Costabile2}.

\subsection{Vandermonde determinants}

We begin by comparing the standard and umbral Wronskians.
\begin{proposition} \label{devw}
$W_{\theta}^{\mathfrak{d}}(t, v)=S_n(t) \boxplus_n W_{\theta}(t,v)$.
\end{proposition}

\begin{proof}
We have 
\[
\langle \langle P_{\ell}, W_{\theta}^{\mathfrak{d}}(-t, v) \rangle \rangle^{\mathfrak{d}}_n = v_{\ell}
\] 
so we get thanks to proposition \ref{devv} that 
\[
\langle \langle P_{\ell}, (-1)^n R_n^{\dagger} \boxplus_n W_{\theta}^{\mathfrak{d}}(-t, v)  \rangle \rangle_n = v_{\ell},
\]
which proves, since the pairing is nondegenerate, that
\[
(-1)^n R_n^{\dagger}(t) \boxplus_n W_{\theta}^{\mathfrak{d}}(-t, v) = W_{\theta}(-t, v).
\]
Hence, $R_n(t) \boxplus_n W_{\theta}^{\mathfrak{d}}(t, v) = W_{\theta}(t, v)$, which gives the result.
\end{proof}

\begin{definition}
Let $\theta$ be a Vandermonde curve. If we write $\theta = T. e$ where $e(t)=(t^k)_{0 \leq k \leq n}$ is the standard curve, we define the volume of $\theta$ in $\mathbb{C}^{\times}$ as $\mathrm{vol}(\theta)=\det T$. 
\end{definition}
In other words, $\mathrm{vol}(\theta)$ is the product of all leading coefficients of the components of $\theta$. Let $dV$ denote the standard volume form $e_0 \wedge \ldots \wedge e_{n+1}$ on $\mathbb{C}^{n+1}$.

\begin{theorem}[generalized Vandermonde theorem] Let $\theta$ be a Vandermonde curve, and $\mathfrak{d}$ be an umbral structure of height $n$. Let $a_1, \ldots, a_n$ are the roots of $S_n$.
\begin{enumerate}
\item[(i)] For any $x_0, \ldots, x_n$ in $\mathbb{C}$, 
\[
\frac{\theta(x_0) \wedge \ldots \wedge \theta(x_n)}{\mathrm{vol}(\theta) \,dV}=  \prod_{i < j} (x_j-x_i) .
\]
\item[(ii)] For any $t$ in $\mathbb{C}$,
\[
\frac{\theta(t+a_1) \wedge \ldots \wedge \theta(t+a_n) \wedge v}{\mathrm{vol}(\theta) \,dV}=  \prod_{i < j} (a_j-a_i) \times n! W_{\theta}^{\mathfrak{d}}(-t, v).
\]
\end{enumerate}
\end{theorem}

\begin{remark}
Remark that the for $v=\theta(t+s)$, 
\[
n! W_{\theta}^{\mathfrak{d}}(-t, v) = S_n(s)=\prod_i (s-a_i)
\] 
and the expression of (ii) reduces to the one in (i).
\end{remark}

\begin{proof}
Let 
\[
R_k(a_1, \ldots, a_k)= \dfrac{\theta(a_1) \wedge \ldots \wedge \theta(a_k)}{\prod_{i < j} (a_j-a_i)},
\]
it is a symetric polynomial in the $a_i's$ with values in $\wedge^{k} \mathbb{C}^{n+1}$, whose components have degree at most $n+1-k$ is each $a_i$. We prove by induction that 

\[
\displaystyle
R_k(t,\ldots, t) = \frac{\theta_t}{0!} \wedge \ldots \wedge \frac{\theta^{(k-1)}_t}{(k-1)!} 
\cdot
\]
We have \[
R_{k+1}(t_1, \ldots, t_k, a)=\frac{R_k(t_1, \ldots, t_k, a)}{(a-t_1) \ldots (a-t_k)}
\] so 
\begin{align*}
R_{k+1}&(t, \ldots, t, a) = \frac{R_{k}(t, \ldots, t) \wedge \theta_a}{(a-t)^k} \\
&= \left(\frac{\theta_t}{0!} \wedge \ldots \wedge \frac{\theta^{(k-1)}_t}{(k-1)!} \right) \wedge 
\frac{\sum_{p=0}^k (a-t)^p \displaystyle \frac{\theta_t^{(p)}}{p!} + (a-t)^{k+1}S(a)}{(a-t)^k} \\
&=\frac{\theta_t}{0!} \wedge \ldots \wedge \frac{\theta^{k}_t}{k!} + (a-t)S(a).
\end{align*}
In particular, $R_{n+1}=\mathrm{vol}(\theta) dV$. We have 
\[
\frac{R_n(t, \ldots, t)\wedge v}{\mathrm{vol}(\theta) dV} = n! W_{\theta}(-t, v).  
\]
so that \[
\dfrac{R_n(-a_1, \ldots, -a_n)\wedge v}{\mathrm{vol}(\theta) dV}
\] 
is the polar form of $W_{\theta}(t, v)$. Now, observe that the map
\[
P \mapsto S_n(t) \boxplus_n P
\]
is obtained by mapping $P$ to $L_P(t-a_1, \ldots, t-a_n)$ where $L_P$ is the polar form of $P$. Indeed, since the two quantities are translation invariant, it suffices to prove that they agree on $t^n$, which is obvious. Hence 
\begin{align*}
W_{\theta}^{\mathfrak{d}}(t, v) &= S_n(t) \boxplus_n W_{\theta}(t, v) \\
&=\dfrac{R_n(-(t-a_1), \ldots, -(t-a_n))\wedge v}{\mathrm{vol}(\theta) dV} \\
&=\dfrac{R_n(a_1-t), \ldots, (a_n-t))\wedge v}{\mathrm{vol}(\theta) dV} \cdot
\end{align*}
\end{proof}

\begin{example}

The determinant
\[
\begin{vmatrix}
(t)_0 & (t+1)_0  & \ldots &(t+n-1)_0  & v_0\\ 
(t)_1 & (t+1)_1  & \ldots &  (t+n-1)_1& v_1\\ 
\vdots & \vdots  & \ldots & \vdots & \vdots \\ 
(t)_n & (t+1)_n  & \ldots & (t+n-1)_n & v_n
\end{vmatrix}
\]
is equal to 
\[
\prod_{k=1}^{n} k! \times \sum_{k=0}^n (-1)^{k} \binom{n}{k}v_{n-k} t^{(k)}.
\]

For the umbral structure $\Delta$, the binomial curve is $\eta = ((t)_k)_{0 \leq k \leq n}$. Since
$S_n(t)=(t)_n$, its roots are $0, 1, 2, \ldots, n-1$. The curve $\eta$ has volume $1$. Besides,
\[
\prod_{i<j} (a_j-a_i) = \prod_{0 \leq i<j \leq n-1} (j-i) = \prod_{k=1}^{n-1} k!
\]
so the determinant equals
\[
\prod_{k=1}^{n} k! \times \eta(-t) \star v
\]
which is exactly the required expression.
\end{example}

\section{Application to finite differences}
In this section, we focus on the case where the umbral structure $\mathfrak{d}$ is given by the forward difference operator $\Delta$ at any level. Hence :
\begin{enumerate}
\item[--] As a formal differential operator, 
$
\Delta = \sum_{p=1}^{+\infty} \dfrac{1}{p!} \dfrac{\partial^p}{\partial t^p}\cdot
$ \\
\item[--]  The binomial sequence is $(S_k(t))_{k \geq 0}=((t)_k)_{k \geq 0}$.
\item[--] $\Delta^* = \Lambda$ is the backward difference operator given by 
\[
\Lambda P(t)=P(t)-P(t-1).
\]
\item[--] $\langle A, B \rangle_n^{\Delta}= \sum_{p=0}^n (-1)^{n-p} \Delta^{n-p} A \times \Lambda^p B$.
\end{enumerate}

\subsection{Roots of the deviation}

The deviation polynomials are symmetric with respect to the vertical line $\mathrm{Re}\,(z)=-\dfrac{n-1}{2}\cdot$

\begin{proposition} \label{drift}
Let $n$ be a positive integer. Then 
\[
R_n(-(n-1)-t)=(-1)^n R_n(t).
\]
\end{proposition}

\begin{proof}
Let $K$ be the polar form of $R_n$ and $L$ be the one of $(-1)^n R_n(-(n-1)-t)$. Then 
\[
K(t, t-1, \ldots, t-(n-1))=t^n.
\]
and 
\begin{align*}
L(t, t-1, \ldots, t-(n-1))&=(-1)^n K(-(n-1)-t, -(n-1)-t+1, \ldots, -t) \\
&=(-1)^n K(-t, -t-1, \ldots, -t-(n-1)) \\
&=(-1)^n (-t)^n \\
&= t^n
\end{align*}
so both polynomials are equal.
\end{proof}

The deviation polynomials admit a closed-form expression

\begin{theorem} \label{demon}
The polynomials $(R_n)_{n \geq 0}$ are given by the formula
\[
R_n = 
\underbrace{\frac{\Delta t^{n+1}}{n+1} \boxplus_n \ldots \boxplus_n \frac{\Delta t^{n+1}}{n+1}}_{n \, \textrm{times}} \boxplus_n \frac{\Lambda t^{n+1}}{n+1} 
=\frac{n! \Delta^n (\Lambda t^{2n+1})}{(2n+1)!}\cdot
\]
In particular, all roots of the deviation polynomial $R_n$ lie on the vertical line $\mathrm{Re}\,(z)=-\left(\dfrac{n-1}{2}\right)\cdot$
\end{theorem}

\begin{proof}
The first part follows directly from theorem \ref{lagrange} (see example \ref{yoo} (i)). The second part is an immediate consequence of Walsh representation theorem. 
\end{proof}

\subsection{Finite difference Grace and Walsh theorems}
For any real numbers $a$ and $b$ with $-\infty \leq a \leq b \leq + \infty$, we put 
\[
V(a, b)=\{z \in \mathbb{C}, a \leq \mathrm{Re}\,(z) \leq b \}.
\]
For any polynomial $P$, we denote by $V(P)$ the smallest strip $V(a, b)$ containing all roots of $P$.
\begin{theorem}[Umbral Grace theorem] \label{hammer}
Let $P$ and $Q$ be two polynomials of degree $n$ such that  $\langle P, Q \rangle_n^{\Delta} =0$. Then 
\[
V(P) \cap \left(V(Q)+\frac{n-1}{2}\right)  \neq \emptyset.
\]
\end{theorem}

\begin{proof}
Let us write $V(Q)=V(a,b)$. By proposition \ref{drift} (ii), $R_n \boxplus_n P$ is $n$-apolar to $Q$, so thanks to Grace theorem applied to the circular domain 
$
\{z \in \mathbb{C}, \mathrm{Re}\,(z) \leq b\}
$, it admits a root $\alpha$ of real part less than or equal to $b$. Using theorem \ref{demon}, since the circular domain
\[
\Omega = \left\{z \in \mathbb{C}, \mathrm{Re}\,(z) \geq -\frac{n-1}{2}\right\}
\]
contains all roots of $R_n$, we can apply Walsh representation theorem (theorem \ref{Walsh}) : $\alpha$ can be written as $z + \omega$ where $z$ is a root of $P$ and $\omega$ lies in $\Omega$. Hence
\[
\mathrm{Re}\,(z) = \mathrm{Re}\,(\alpha)-\mathrm{Re}\,(\omega) \leq b + \frac{n-1}{2}
\]
By the same argument, $P$ admits a root whose real part is greater than or equal to $a+\dfrac{n-1}{2}\cdot$ The result follows. 
\end{proof}

\begin{remark}
If $P(\alpha)=0$, $\langle \langle P(t), (t-\alpha)^{(n)} \rangle\rangle_n=0$. We have
\[
V(Q)+\frac{n-1}{2}=\left[\mathrm{Re}(\alpha)-\frac{n-1}{2}, \mathrm{Re}(\alpha)+\frac{n-1}{2} \right]
\]
and $\mathrm{Re}\,(\alpha)$ is an element of $V(P)$ in the middle of this interval, which is consistant theorem \ref{hammer}.
\end{remark}

\begin{remark}
It can very well happen than $P$ has no root at all inside the strip $V(Q)+\dfrac{n-1}{2}\cdot$ An explicit example is provided by the polynomial $P(t)=t^2+(2c-1)t-c$, where $c$ is some nonzero real number. 
\begin{align*}
R_2(t) \boxplus_2 P(t) &= \left(t^2+t+\dfrac{1}{2} \right) \boxplus_2 (t^2+(2c-1)t-c) \\
&=t^2+(2c-1)t-c+\frac{1}{2}(2t+2c-1)+1 \\
&=t^2+2ct
\end{align*}
so $ \langle P, t^2 \rangle_2^{\Delta}=0$. However the roots of $P$ are never purely imaginary. 
\end{remark}

\begin{remark}
Grace classical theorem does not apply for a bounded strip (which is not a circular domain). However, if $V(Q)=V(a, b)$, and $\langle \langle P, Q \rangle \rangle_n=0$, theorem \ref{hammer} can be rephrased by saying that $P$ has a root in each of the two circular domains $V \left(-\infty, b+\dfrac{n-1}{2}\right)$ and $V \left(a+\dfrac{n-1}{2}, +\infty\right)$. Note that this root need not be the same for both half-planes. Thus, Theorem \ref{hammer} constitutes a natural extension of Grace's theorem to vertical half-planes.
\end{remark}

\begin{theorem}[Umbral Walsh theorem] \label{WC}
Let $P$ and $Q$ be polynomials in $\mathbb{C}_n[t]$. Then 
\[
V(P \boxplus_n^{\mathfrak{d}} Q) \subseteq V(P)+V(Q)-\left(\dfrac{n-1}{2}\right)\cdot
\]
\end{theorem}

\begin{proof}
If $z$ is a root of $P \boxplus_n^{\mathfrak{d}} Q$, $\langle P(t), Q(z-t) \rangle_n^{\Delta}=0$ so by theorem \ref{hammer}, 
\[
V(P) \cap \left(\mathrm{Re}\,(z)-V(Q) + \frac{n-1}{2} \right) \neq \emptyset.
\]
This implies that $\mathrm{Re}(z)$ belongs to $V(P)+V(Q)-\left(\dfrac{n-1}{2}\right)\cdot$
\end{proof}

\begin{remark}
The constant $\dfrac{n-1}{2}$ is optimal, as can be seen from theorem \ref{demon} since $
t^n \boxplus_n^{\mathfrak{d}} t^n = R_n$.
\end{remark}

\nocite{*}
\bibliographystyle{plain}
\bibliography{biblio}

\end{document}